\documentclass[11pt,a4paper]{article}
\usepackage[T1]{fontenc}
\usepackage[utf8]{inputenc}
\usepackage{authblk}
\usepackage{amssymb, amsmath}
\usepackage{amsthm, amsfonts,mathrsfs}
\usepackage[T1]{fontenc}
\usepackage{times}
\usepackage{graphicx}
\usepackage{subfigure}
\usepackage{cite}
\usepackage{hyperref}
\usepackage{epsfig}
\usepackage{color}
\usepackage[all]{xypic}
\hypersetup{colorlinks=true,citecolor=blue,linkcolor=blue}
\newtheorem{thm}{Theorem}[section]
\newtheorem{coro}[thm]{Corollary}

\newtheorem{prop}[thm]{Proposition}
\newtheorem{remark}[thm]{Remark}

\numberwithin{equation}{section}
\newcommand{\ep}{\epsilon}

\newcommand{\al}{\alpha}

\newcommand{\intot}{\int_{0}^{t}}
\newcommand{\intos}{\int_{0}^{s}}

\newcommand{\supt}{\sup\limits_{0\leq t\leq T}}
\newcommand{\supe}{\sup\limits_{0<\epsilon\leq 1}}

\newcommand{\sxm}{X^\epsilon}

\newcommand{\xe}{X^\epsilon}
\newcommand{\ye}{Y^\epsilon}
\newcommand{\xei}{X^{\epsilon,i}}
\newcommand{\xel}{X^{\epsilon,l}}
\newcommand{\yei}{Y^{\epsilon,i}}
\newcommand{\yej}{Y^{\epsilon,j}}
\newcommand{\yek}{Y^{\epsilon,k}}
\newcommand{\yel}{Y^{\epsilon,l}}

\newcommand{\mi}{M^{-1}}
\newcommand{\phie}{\phi^\epsilon}
\newcommand{\zr}{z^{\epsilon,r}}
\title{\textbf{On the Approximation of Differential Equations Driven by Some  Random Processes  as Rough Paths}\thanks{This work is supported by NSFC Grant No.  12371243.}}

\author{Qingming Zhao\thanks{qingming.zhao@smail.nju.edu.cn}~~}
\author{Xueru Liu\thanks{dg21210009@smail.nju.edu.cn}~~}
\author{~Wei Wang\thanks{Corresponding author:wangweinju@nju.edu.cn}}
\affil{School of Mathematics, Nanjing University, Nanjing 210093, P. R. China}

\date{}

\begin{document}
	\maketitle

	\noindent{\small{\hspace{1.1cm} }}

	\noindent \textbf{Abstract~~~}   We explore the limit of stochastic differential equations driven by some random processes satisfying singularly perturbed second order stochastic differential equations. The main tool we employ is the universal limit theorem in rough path theory. To this end, we lift the  random process as a rough path in a natural manner. After suitable change-of-variable, the random process has a form of slow-fast system. Moment estimates of both the random process and its lift are given, followed by which, averaging technique and convergence theorem in rough path topology are used to identify the limit.
	\\[2mm]
	\textbf{Keywords~~~} {Smoluchowski--Kramers approximation}, {Averaging}, {Singular perturbation}, {Rough path theory}
	\\[2mm]
	\\
	\textbf{2020 Mathematics Subject Classification~~~}60L90, 60H10

	\section{Introduction}
	Let us consider approximation of stochastic differential equations (SDE) of the form
	\begin{equation}\label{randomode}
		dZ^\ep_t=b(Z^\ep_t)dt+\sigma(Z^\ep_t)dX^\ep_t,
	\end{equation}
	 where $X^\ep$ is some random process converging to  $X$ in some sense. At first glance, one might expect that $Z^\ep\to Z$ with
	  \begin{equation*}
	 	dZ_t=b(Z_t)dt+\sigma(Z_t)dX_t.
	 \end{equation*}
	  However, this is not true in general, as seen in the classical Wong--Zakai approximation, and there is an additional term called It\^o--Stratonovich correction, appearing in the limit equation\cite{WONGZAKAI}.
	
	 In the pioneering work \cite{FGL15}, they assumed that the process $\xe$ is governed via the following Langevin equation
	\begin{equation}\label{simplelangevin}
		\ep^2\ddot{X}^\ep_t+M\dot{X}^\ep_t=\dot{W}_t,\quad \xe_0=0,\enspace\dot{X}^\ep_0=0.
	\end{equation}
	The process $X^\ep$ is also called {\it physical Brownian motion}~(pBm). 
	This equation describes the motion of a particle with mass $\ep^2$ under a random external force $W$, classically assumed to be a standard (mathematical) Brownian motion, and $M$ is the coefficient of friction. They firstly showed that $\xe\to X:=M^{-1}W$, which can be formally obtained by setting $\ep=0$ in (\ref{simplelangevin}), and this type of convergence is known as Smoluchowski--Kramers (SK) approximation\cite{CF06A,CF06B,LW25, chaos24, FGL15}. Then, they found that, the solution $Z^\ep$ of (\ref{randomode}), converges to a non-trivial limit. The reason lies in the fact that, in the context of rough path theory, the It\^o lift of $\xe$ does not converge to the It\^o lift of $X$.

	However, in the physical model, it commonly happens that the friction coefficient $M$ is not a constant, and there is a non-linear potential term $F$.   In this paper, we explore the limit of (\ref{randomode}) with the driving signal $\xe$  modeled by a more general stochastic differential equation 
	\begin{equation}\label{SKm}
		\ep^2\ddot{X}^\ep_t+M(X^\ep_t)\dot{X}^\ep_t=F(X^\ep_t)+\dot{W}_t,\enspace X^\ep_0=0,\enspace \dot{X}^\ep_0=0.
	\end{equation}
	We call it a general   physical Brownian motion. 
	Setting the velocity $V^\ep_t=\dot{X}^\ep_t$ and momentum $P^\ep_t=\ep^2 V^\ep_t$ respectively,  (\ref{SKm}) is equivalent to 
	\begin{equation}\label{SKmv}
		\begin{cases}
			d\sxm_t=V^\ep_tdt,\enspace \sxm_0=0,\\
			dV^\ep_t=-\frac{1}{\ep^2}M(\sxm_t)V^\ep_tdt+\frac{1}{\ep^2}F(\sxm_t)dt+\frac{1}{\ep^2}dW_t,\enspace V^\ep_0=0,
		\end{cases}
	\end{equation}
	and
	\begin{equation}\label{SKmp}
		\begin{cases}
			d\sxm_t=\frac{1}{\ep^2}P^\ep_tdt,\enspace \sxm_0=0,\\
			dP^\ep_t=-\frac{1}{\ep^2}M(\sxm_t)P^\ep_tdt+F(\sxm_t)dt+dW_t,\enspace P^\ep_0=0.
		\end{cases}
	\end{equation} 
	We also introduce $Y^\ep_t:=\frac{1}{\ep}P^\ep_t$~\cite[Page 7944]{FGL15}, so that
	\begin{equation}\label{SKxy}
		\begin{cases}
			d\xe_t=\frac{1}{\ep}\ye_tdt,\enspace\xe_0=0,\\
			d\ye_t=-\frac{1}{\ep^2}M(\xe_t)\ye_tdt+\frac{1}{\ep}F(\xe_t)dt+\frac{1}{\ep}dW_t,\enspace\ye_0=0.
		\end{cases}
	\end{equation}
	The system (\ref{SKxy}) has a form of slow-fast system \cite{DW14}, making it possible to find its limit by averaging method. 
	We always  assume that equations (\ref{SKm}) and (\ref{randomode}) are defined on a complete stochastic basis $(\Omega,\mathcal{F},(\mathcal{F}_t)_{0\leq t\leq T},P)$ with $0<T<\infty$ fixed, and $W$ is a $d$-dimensional standard Brownian motion. Hypothesis on other coefficients are detailed in Section \ref{section-preli}. We also remark that, since $\xe$ is of finite variation which can be seen from (\ref{SKmv}) , (\ref{randomode}) is in fact a random ODE. 
	
	The universal limit theorem (also known as continuity of It\^o--Lyons map) in rough path theory~\cite{Lyo98} indicates that, in order to identify the limit of (\ref{randomode}), it is necessary to find the limit of the lift $\mathbf{X}^\ep=(\xe,\mathbb{X}^\ep)$ of $\xe$, detailed in Section \ref{section-preli}. To this end, we will firstly calculate the limit of $\xe$. In contrast to the constant friction coefficient case, we cannot obtain the limit by simply dropping the small mass term $\ep^2\ddot{X}^\ep$ in (\ref{SKm}). In fact, Hottovy et al.\cite{HMV15} showed   that, under certain assumptions,
	$$\lim_{\ep\rightarrow 0}\mathbb{E}\supt|X^\ep_t-X_t|^2=0,$$ 
	where
	\begin{equation}\label{level1limit}
		dX_t=S(X_t)dt+\mi(X_t)F(X_t)dt+\mi(X_t)dW_t,\enspace X_0=0,
	\end{equation}
	and
	\begin{equation}\label{sjdef}
		S_j(x)=\sum_{k,l}(\partial_{x_l}\mi)_{jk}(x)J_{kl}(x),
	\end{equation} 
	with $J$ the solution of Lyapunov equation
	\begin{equation}\label{lyapunov}
		M(x)J(x)+J(x)M^\top(x)=\mathrm{I}_d.
	\end{equation} 
	Here $\mathrm{I}_d$ is the identity matrix in $\mathbb{R}^{d\times d}.$ For more research in SK approximation in position-dependent friction case, we refer to \cite{CX22, HOTT13, LW23, WW23, SWW24}. As we see in (\ref{SKxy}), SK approximation is naturally linked to stochastic averaging, initiated by Khaminskii \cite{KHAMINSKII}. Numerous work concerning about stochastic averaging appeared in recent years, see e.g. \cite{LGQ25,WRD12, XDX20, XLL21}. There is also research about averaging principle in rough path framework, see for example in \cite{LLP25,BIX23}.   Therefore, unlike the existing literature, we use a stochastic averaging method to obtain the limit in a much stronger sense. The next step is to identify the limit of $\mathbb{X}^\ep$, which is more delicate. After these work, we are able to find the effective approximation of (\ref{randomode}) by utilizing universal limit theorem.

	Our current article differs from these existing literature in at least three aspects. The first is that, our main theorems show the convergence of the lifted process $\mathbf{X}^\ep$ of $\xe$ in the sense of $p$-th moment of $\al$-H\"older inhomogeneous rough path distance with $1/3<\al<1/2$. This is much stronger than \cite{HMV15} and \cite{SWW24}, in which they showed the convergence of the original process $\xe$ in $L^2(\Omega;C([0,T];\mathbb{R}^d))$ and convergence in distribution, respectively. The second is that, in \cite{HMV15}, they essentially assume the tightness of $(\xe)_{0<\ep\leq 1}$ in $C([0,T];\mathbb{R}^d).$ In our paper, we drop this assumption, and obtain a much stronger result, that is, $(\mathbf{X}^\ep)_{0<\ep\leq 1}$ is tight in  $\mathscr{C}^{0,\al}([0,T];\mathbb{R}^d)$, the closure of the smooth rough paths in $\al$-H\"older rough path space $\mathscr{C}^{\al}([0,T];\mathbb{R}^d)$, which is a corollary of the boundedness of $(\mathbf{X}^\ep)_{0<\ep\leq 1}$ in $\mathscr{C}^{\al}([0,T];\mathbb{R}^d)$. In \cite{FGL15}, they showed the boundedness by utilizing the Gaussian property of $V^\ep$ in the constant friction case. However, in our setting, $V^\ep$ is non-Gaussian due to the fact that the friction $M$ depends on the position, which makes the proof more challenging. The third is that, as a consequence of the variable friction again, $V^\ep$ is non-Markovian, so we cannot pass the limit by a direct application of ergodic theorem when calculating the limit of $\mathbb{X}^\ep$ as in \cite{FGL15}. To fix this, we prove a strong averaging principle for a specific slow-fast system. We will see in Section \ref{section-aver} that our slow system is only locally Lipschitz, which does not satisfy the commonly assumed condition in the aforementioned reference. Thanks to the structure of our slow system, we employ Poisson equation technique\cite{PAR2001,PAR2003, PAR2005} to show the convergence in $L^p$ sense.

	 The rest of this paper is organized as follows. In Section \ref{section-preli}, we recall some basic rough path theory, list some notations, and present our main results. In Section \ref{section-estimate}, we give moment estimate in rough path norm. In Section \ref{section-aver}, we display an averaging principle for specific slow-fast system, which is used to pass the limit. After these preparation, we prove the main theorems in Section \ref{section-proof}.

	\section{Preliminary and Main Results}\label{section-preli}
	
 We firstly recall some basic definitions of rough path theory~\cite[e.g.]{FH20}.  For systematic display of rough path theory, we refer to \cite{LQ02,LCL07,FV10,FH20}. 
 
 Let $X:[0,T]\to \mathbb{R}^d,$ $\mathbb{X}:\Delta_{[0,T]}\to\mathbb{R}^{d\times d}$, where $$\Delta_{[0,T]}:=\{(s,t)\in [0,T]^2:s\leq t\}.$$ Given $\frac{1}{3}<\al\leq\frac{1}{2},$  a map $\mathbf{X}=(X,\mathbb{X})$ is called an $\al$-H\"older rough path over $\mathbb{R}^d$, if the following conditions hold:
	
	(i) $||X||_\al:=\sup_{0\leq s\leq t\leq T}\limits\frac{|X_t-X_s|}{|t-s|^\al}<\infty,$
	
	(ii) $||\mathbb{X}||_{2\al}:=\sup_{0\leq s\leq t\leq T}\limits\frac{|\mathbb{X}_{s,t}|}{|t-s|^{2\al}}<\infty,$
	
	(iii) (Chen relation) $\mathbb{X}_{s,u}=\mathbb{X}_{s,t}+\mathbb{X}_{t,u}+X_{s,t}\otimes X_{t,u}, \enspace 0\leq s\leq t\leq u\leq T,$
	
	\noindent where $X_{s,t}:=X_t-X_s$ for $(s,t)\in \Delta_{[0,T]}.$ 
	
	The set of all $\al$-H\"older rough path over $\mathbb{R}^d$ is denoted by $\mathscr{C}^\al([0,T];\mathbb{R}^d).$ Note that this space is not a linear space due to the nonlinear Chen relation. Given $\mathbf{X},\mathbf{Y}\in\mathscr{C}^\al([0,T];\mathbb{R}^d),$ the inhomogeneous $\al$-H\"older rough path metric is defined as 
	$$\rho_\al(\mathbf{X},\mathbf{Y}):=||X-Y||_\al+||\mathbb{X}-\mathbb{Y}||_{2\al}=\sup_{0\leq s\leq t\leq T}\frac{|X_{s,t}-Y_{s,t}|}{|t-s|^\al}+\sup_{0\leq s\leq t\leq T}\frac{|\mathbb{X}_{s,t}-\mathbb{Y}_{s,t}|}{|t-s|^{2\al}}.$$
	This metric is inhomogeneous with respect to the dilation operator $\delta_\lambda$ ($\lambda\in\mathbb{R}$), defined by $\delta_\lambda ((a,b)):=(\lambda a,\lambda^2 b),\enspace a\in \mathbb{R}^d, b\in\mathbb{R}^{d\times d}.$ Besides, the $\al$-H\"older homogeneous rough path norm is defined by $$|||\mathbf{X}|||_\al:=||X||_\al+\sqrt{||\mathbb{X}||_{2\al}}.$$

	Next we recall concepts of rough differential equations (RDE). We adopt the definition in \cite[Section 8.7]{FH20}.  Let $b\in C^3_b(\mathbb{R}^e;\mathbb{R}^e)$, $\sigma\in C^3_b(\mathbb{R}^e;\mathbb{R}^{e\times d})$, that is, they are $n$ times continuously differentiable with bounded $k$-th derivative, $k=1,2,3.$  By a solution of the RDE
	\begin{equation}\label{basicrde}
		dZ_t=b(Z_t)dt+\sigma(Z_t)d\mathbf{X}_t,\quad Z_0=\xi,
	\end{equation}
	 we mean that
	$$Z_t-Z_s=b(Z_s)(t-s)+\sigma(Z_s)X_{s,t}+D\sigma(Z_s)\sigma(Z_s)\mathbb{X}_{s,t}+R_{s,t}, \quad (s,t)\in\Delta_{[0,T]},$$
	 and $$\lim\limits_{|\mathcal{P}|\to 0}\sup\limits_{\mathcal{P}\subset [0,T]}\sum\limits_{[s,t]\in\mathcal{P}}|R_{s,t}|=0,$$ here $\mathcal{P}\subset [0,T]$ means $\mathcal{P}$ is a partition of $[0,T]$, and $|\mathcal{P}|$ is the mesh of $\mathcal{P}$. It turns out that, the It\^o--Lyons map of (\ref{basicrde}) is locally Lipschitz with respect to $\rho_\al$\cite[Thereom 8.5]{FH20}. To be precise, if $\tilde{Z}$ is the solution of
	 \begin{equation}
	 	d\tilde{Z}_t=b(\tilde{Z}_t)dt+\sigma(\tilde{Z}_t)d\tilde{\mathbf{X}}_t,\quad \tilde{Z}_0=\tilde{\xi},
	 \end{equation}
	and $M$ is such that $\max\{|||\mathbf{X}|||_\al,|||\tilde{\mathbf{X}}|||_\al\}\leq M,$ then there exists a constant $C=C(\al,b,\sigma,M)$ such that $$||Z-\tilde{Z}||_\alpha\leq C(|\xi-\tilde{\xi}|+\rho_\al(\mathbf{X},\mathbf{\tilde{X}})).$$

	With preliminaries above, we are ready to display our main results. For each $0<\ep\leq 1,$ define the second-order process by
	\begin{equation}\label{second-order-process}
		\mathbb{X}^\ep_{s,t}:=\int_s^t \xe_{s,u}\otimes d\xe_u,\quad (s,t)\in \Delta_{[0,T]}.
	\end{equation}  Since $X^\ep$ is of finite variation, the integral above is in the sense of Riemann--Stieltjes. Setting $\mathbf{X}^\ep=(\xe,\mathbb{X}^\ep),$ it is straightforward to check that $\mathbf{X}^\ep$ satisfies Chen relation, and that $\mathbf{X}^\ep$ lies in $\mathscr{C}^\al([0,T];\mathbb{R}^d)$ for all $\frac{1}{3}<\al<\frac{1}{2}.$ 
	
	
We impose the following assumptions:
	
	\noindent($\mathbf{A1}$) There exists a constant $\lambda>0$, such that the smallest eigenvalue $\lambda_1$ of $\frac{M+M^\top}{2}$ satisfies that $$\lambda_1(x)\geq \lambda, \quad x\in\mathbb{R}^d.$$
	
	\noindent($\mathbf{A2}$) $M\in C^2(\mathbb{R}^d;\mathbb{R}^{d\times d})$, and $\mi\in C^3(\mathbb{R}^d;\mathbb{R}^{d\times d})$.
	
	\noindent($\mathbf{A3}$) $M$, $\mi$ and $\partial_{x_j}\mi_{ik}$ are globally Lipschitz, and $\mi$ is bounded, $1\leq i,j,k\leq d$.
	
	\noindent($\mathbf{A4}$) $F$ is globally Lipschitz and bounded.
	\begin{remark}
		From ($\mathbf{A1}$), one finds that
		$$y^\top M(x)y\geq \lambda|y|^2,\quad (x,y)\in\mathbb{R}^d\times\mathbb{R}^d,$$ which is essential in moment estimates. ($\mathbf{A2}$) is used to obtain  regularity of the solution of Poisson equation in Section \ref{section-aver}. Well-posedness of (\ref{level1limit}) and (\ref{SKxy}) are guaranteed by ($\mathbf{A3}$) and ($\mathbf{A4}$). 
	\end{remark}
	
	Now we state our main theorems.
	\begin{thm}\label{maintheorem}
		Let $X$ be the solution of (\ref{level1limit}). Define
		\begin{equation}\label{defofbbx}
			\mathbb{X}_{s,t}:=\int_s^t X_{s,u}\otimes\circ dX_u+\frac{1}{2}\int_s^t J(X_{u})(\mi)^\top(X_{u})-\mi(X_{u})J(X_{u})du,
		\end{equation}
		where $J$ is the solution of (\ref{lyapunov}), and $\otimes\circ$  denotes  the Stratonovich integral.  Under assumptions $(\mathbf{A1})$--$(\mathbf{A4})$, for each $\frac{1}{3}<\al<\frac{1}{2}$ and $1\leq p<\infty,$
		$$\lim\limits_{\ep\to0}\mathbb{E}[\rho_\alpha(\mathbf{X}^\ep,\mathbf{X})^p]=0,$$ where $\mathbf{X}=(X,\mathbb{X}).$ 
	\end{thm}
\begin{remark}
By Kolmogorov criterion for rough paths\cite[Theorem 3.1]{FH20}, one verifies that $\mathbf{X}$ lies in $\mathscr{C}^\al([0,T];\mathbb{R}^d)$ for all $\frac{1}{3}<\al<\frac{1}{2}.$
\end{remark}

An application of continuity of It\^o--Lyons map with careful computation, detailed in Section~\ref{section-proof}, yield the following theorem.
\begin{thm}\label{maintheorem2}
	Let $b\in C^3_b(\mathbb{R}^e;\mathbb{R}^e)$ and $\sigma\in C^3_b(\mathbb{R}^e;\mathbb{R}^{e\times d}).$ Consider random ODE
	$$dZ^\ep_t=b(Z^\ep_t)dt+\sigma(Z^\ep_t)dX^\ep_t,\quad Z^\ep_0=\xi.$$Under assumptions $(\mathbf{A1})$--$(\mathbf{A4})$,
	$Z^\ep \to Z$ in $C^\alpha([0,T];\mathbb{R}^e)$ in probability with $Z_0=\xi$ and
	\begin{eqnarray}
		&&dZ^i_t\nonumber\\&=&[b_i(Z_t)+\sum\limits_{j,k,l,k^\prime}\partial_{z_k}\sigma_{ij}(Z_t)\sigma_{kl}(Z_t)J_{lk^\prime}(X_t)\mi_{jk^\prime}(X_t)\nonumber\\&&+\sum_{j}\sigma_{ij}(Z_t)S_j(X_t)+\sum_{j,k}\sigma_{ij}(Z_t)M^{-1}_{jk}(X_t)F_k(X_t)]dt+\sum_{j,k}\sigma_{ij}(Z_t)M^{-1}_{jk}(X_t)dW^k_t,\nonumber
	\end{eqnarray}
	where $S_j$ is defined in (\ref{sjdef}).
\end{thm}
We end this section with a collection of frequently used notation.
\section*{Notation}
\begin{description}
	\item $|\cdot|$: Euclidean norm of either a vector or a matrix.
	
	\item $A^\top$: Transpose of a matrix or a vector $A$.
	
	\item $\mathrm{I}_d$: Identity matrix of $d\times d$.
	
	\item $(x,y)$: Inner product of two vectors $x$ and $y$.
	
	\item $Df, D^2 f$: Gradient and Hessian of a function $f:\mathbb{R}^d\to\mathbb{R}$.
	
	\item $C^k(U;V)$: The set consists of $k$-th continuously differentiable functions from $U$ to $V$, where $k\in\mathbb{N},$ and $U$, $V$ are finite-dimensional Banach spaces.
	
	\item $a\lesssim b$: There exists a constant $C>0$ such that $a\leq Cb.$ Throughout the paper, $C$ never depends on $\ep$. Additionally, in the proof of Proposition \ref{thmforcontinuity}, it neither depends on $s$ nor $t$.
\end{description}
	\section{Moment Estimates}\label{section-estimate}
	In this section, we give several moment estimates for system (\ref{SKxy}). 
\begin{prop}\label{thmformomentestimate}
	Under assumptions ($\mathbf{A1}$)--($\mathbf{A4}$), for each $1\leq p<\infty,$ $$\supe\supt\mathbb{E}|\ye_t|^p<\infty.$$
\end{prop}
\begin{proof}
	 For each $p\geq 2,$ set $\phi(y)=|y|^p,$ then
	$$D\phi(y)=p|y|^{p-2}y,\quad D^2\phi(y)=p(p-2)|y|^{p-4}yy^\top+p|y|^{p-2}\mathrm{I}_d.$$
	 Applying It\^o formula to $\phi(\ye_t)$,
	\begin{eqnarray}
		&&|\ye_t|^p\nonumber\\&=&\intot [D\phi(\ye_s)]^\top d\ye_s+\frac{1}{2}\intot\mathrm{Tr}\Big(\frac{1}{\ep^2}D^2\phi(\ye_s)\Big)ds\nonumber\\&=&-\frac{p}{\ep^2}\intot|\ye_s|^{p-2}(\ye_s)^\top M(\xe_s)\ye_sds+\frac{p}{\ep}\intot|\ye_s|^{p-2}(\ye_s)^\top F(\xe_s)ds+H^\ep_t\nonumber
		\\&&+\frac{1}{2\ep^2}\intot\mathrm{Tr}[p(p-2)|\ye_s|^{p-4}\ye_s(\ye_s)^\top+p|\ye_s|^{p-2}\mathrm{I}_d]ds,
	\end{eqnarray}
	 where $H^\ep$ is a martingale. Taking expectation, differentiating and applying $(\mathbf{A1})$ and $(\mathbf{A4})$,
	\begin{equation}\label{noname2}
		\frac{d}{dt}\mathbb{E}|\ye_t|^p\lesssim-\frac{1}{\ep^2}\mathbb{E}|\ye_t|^p+\frac{1}{\ep^2}\mathbb{E}|\ye_t|^{p-2}.
	\end{equation}
	First for  $p=2$, by the  comparison principle, 
	$$\supe\supt\mathbb{E}|\ye_t|^2<\infty.$$ By (\ref{noname2}) and induction,$$\supe\supt\mathbb{E}|\ye_t|^p<\infty.$$ for all even $p$, and hence for all $1\leq p<\infty.$
	\end{proof}

The next proposition concerns about H\"older norm of the path $\xe$ and its lift.
\begin{prop}\label{thmforcontinuity} Under assumptions ($\mathbf{A1}$)--($\mathbf{A4}$), there is a constant $C$ independent of $\ep$, such that
	
	\noindent(i) For each $1\leq p<\infty,$ $$\mathbb{E}|\xe_t-\xe_s|^p\leq C|t-s|^{p/2},\quad 0\leq s\leq t\leq T.$$
	
	\noindent(ii) For each $1\leq p<\infty,$ $$\mathbb{E}\Big|\int_{s}^{t}\xe_{s,u}\otimes d\xe_u\Big|^p\leq C|t-s|^p,\quad 0\leq s\leq t\leq T.$$
	\end{prop}
	\begin{proof}
		We give estimates by dichotomy.
		From (\ref{SKxy}), $$\ep d\ye_t=-\frac{1}{\ep}M(\xe_t)\ye_tdt+F(\xe_t)dt+dW_t=-M(\xe_t)d\xe_t+F(\xe_t)dt+dW_t.$$ Rearranging,
		\begin{eqnarray}\label{expressionfordxe}
			d\xe_t=\mi(\xe_t)dW_t+\mi(\xe_t)F(\xe_t)dt-\ep\mi(\xe_t)d\ye_t,
		\end{eqnarray} 
		
		{\bfseries Case 1:} $\ep\geq\sqrt{t-s}$. 
		
		By (\ref{SKxy}), H\"older inequality, Fubini theorem and Proposition \ref{thmformomentestimate},
		\begin{eqnarray}
			\mathbb{E}|\xe_t-\xe_s|^p&\leq&\frac{1}{\ep^p}\mathbb{E}\Big(\int_s^t|\ye_u|du\Big)^p\nonumber\\&\leq&\frac{1}{\ep^p}(t-s)^{p-1}\int_s^t\mathbb{E}|\ye_u|^pdu\nonumber\\&\leq&\frac{1}{\ep^p}(t-s)^p\leq (t-s)^{p/2}.
		\end{eqnarray}
		
		{\bfseries Case 2:} $\ep<\sqrt{t-s}$.

From (\ref{expressionfordxe}),
\begin{eqnarray}\label{incrementofxe}
	&&\xe_t-\xe_s\nonumber\\&=&\int_s^t \mi(\xe_u)dW_u+\int_s^t\mi(\xe_u)F(\xe_u)du-\ep\int_s^t\mi(\xe_u)d\ye_u\nonumber\\&=:&I^{\ep,1}(s,t)+I^{\ep,2}(s,t)+I^{\ep,3}(s,t).
\end{eqnarray}
By Burkholder--Davis--Gundy inequality and ($\mathbf{A4}$),
\begin{equation}\label{I1st}
	\mathbb{E}|I^{\ep,1}(s,t)|^p\lesssim\mathbb{E}\Big(\int_s^t|\mi(\xe_u)|^2du\Big)^{p/2}\lesssim (t-s)^{p/2}.
\end{equation}
\begin{equation}\label{I2st}
	\mathbb{E}|I^{\ep,2}(s,t)|^p\lesssim\mathbb{E}\Big(\int_s^t|\mi(\xe_u)F(\xe_u)|du\Big)^p\lesssim(t-s)^p\lesssim(t-s)^{p/2}.
\end{equation}
Applying integration-by-part formula and (\ref{SKxy}),
\begin{eqnarray}\label{preestforI3st}
	&&I^{\ep,3}_i(s,t)\nonumber\\&=&\ep\sum_j\int_s^t\mi_{ij}(\xe_u)d\yej_u\nonumber\\&=&\ep\sum_j[\mi_{ij}(\xe_t)\yej_t-\mi_{ij}(\xe_s)\yej_s]-\ep\sum_{j,l}\int_s^t\yej_ud\mi_{ij}(\xe_u)\nonumber\\&=&\ep\sum_j[\mi_{ij}(\xe_t)\yej_t-\mi_{ij}(\xe_s)\yej_s]-\ep\sum_{j,l}\int_s^t\yej_u\partial_{x_l}\mi_{ij}(\xe_u)d\xel_u\nonumber\\&=&\ep\sum_j[\mi_{ij}(\xe_t)\yej_t-\mi_{ij}(\xe_s)\yej_s]-\sum_{j,l}\int_s^t\yej_u\yel_u\partial_{x_l}\mi_{ij}(\xe_u)du\nonumber\\&=:&J^{\ep,1}_i(s,t)+J^{\ep,2}_i(s,t).
\end{eqnarray}
Since
\begin{equation*}
	J^{\ep,1}_i(s,t)=\ep\sum_j[\mi_{ij}(\xe_t)(\yej_t-\yej_s)+\yej_s(\mi_{ij}(\xe_t)-\mi_{ij}(\xe_s))],
\end{equation*}
an application of ($\mathbf{A3}$), Proposition \ref{thmformomentestimate} and the dichotomy assumption yields
\begin{eqnarray}\label{J1st}
		&&\mathbb{E}|J^{\ep,1}_i(s,t)|^p\nonumber\\&\lesssim&\ep^p\mathbb{E}(|\mi_{ij}(\xe_t)|^p|\yej_t-\yej_s|^p)+\ep^p\mathbb{E}(|\yej_s|^p|\mi_{ij}(\xe_t)-\mi_{ij}(\xe_s)|^p)\nonumber\\&\lesssim&\ep^p\leq(t-s)^{p/2}.
\end{eqnarray}
 By H\"older inequality, ($\mathbf{A3}$) and Proposition \ref{thmformomentestimate},
 \begin{eqnarray}\label{J2st}
 		\mathbb{E}|J^{\ep,2}_i(s,t)|^p&\lesssim&\mathbb{E}\Big(\int_s^t|\yej_u\yel_u\partial_{x_l}\mi_{ij}(\xe_u)|du\Big)^p\nonumber\\&\lesssim&\mathbb{E}\Big(\int_s^t|\yej_u\yel_u|du\Big)^p\nonumber\\&\leq&(t-s)^{p-1}\int_s^t\mathbb{E}|\yej_u\yel_u|^pdu\nonumber\\&\lesssim&(t-s)^p\lesssim(t-s)^{p/2}.
 \end{eqnarray}
 Combining (\ref{preestforI3st})--(\ref{J2st}),
 \begin{equation}\label{I3st}
 	\mathbb{E}|I^{\ep,3}_i(s,t)|^p\lesssim(t-s)^{p/2}.
 \end{equation}
 Combining (\ref{incrementofxe})--(\ref{I2st}) and (\ref{I3st}), we have
 \begin{equation}
 	\mathbb{E}|\xe_t-\xe_s|^p\lesssim (t-s)^{p/2},
 \end{equation}
 and this finishes the proof of (i). Next we turn to (ii).
 
 {\bfseries Case 1:} $\ep\geq\sqrt{t-s}$.
 
 From (\ref{SKxy}) we see that $$\mathbb{X}^\ep_{s,t}=\frac{1}{\ep}\int_s^t\xe_{s,u}\otimes\ye_udu.$$ By H\"older inequality, Proposition \ref{thmformomentestimate} and Proposition \ref{thmforcontinuity} (i),
 \begin{eqnarray}
 	\mathbb{E}|\mathbb{X}^\ep_{s,t}|^p&\leq&\frac{1}{\ep^p}(t-s)^{p-1}\int_s^t\mathbb{E}|\xe_{s,u}\otimes\ye_u|^pdu\nonumber\\&\leq&\frac{1}{\ep^p}(t-s)^{p-1}\int_s^t\sqrt{\mathbb{E}|\xe_{s,u}|^{2p}}\sqrt{\mathbb{E}|\ye_u|^{2p}}du\nonumber\\&\lesssim&\frac{1}{\ep^p}(t-s)^{p-1}\int_s^t(u-s)^{p/2}du\nonumber\\&\lesssim&\frac{1}{\ep^p}(t-s)^{3p/2}\leq(t-s)^p.
 \end{eqnarray}

 {\bfseries Case 2:} $\ep<\sqrt{t-s}$.
 
 By (\ref{expressionfordxe}),
 \begin{eqnarray}\label{preest-increment2}
 	&&\mathbb{X}^\ep_{s,t}\nonumber\\&=&\int_s^t\xe_{s,u}\otimes [\mi(\xe_u)dW_u]+\int_s^t \xe_{s,u}\otimes [\mi(\xe_u)F(\xe_u)du]-\ep\int_s^t\xe_{s,u}\otimes [\mi(\xe_u)d\ye_u]\nonumber\\&=:&I^{\ep,1}(s,t)+I^{\ep,2}(s,t)+I^{\ep,3}(s,t).
 \end{eqnarray}
 By Burkholder--Davis--Gundy inequality, ($\mathbf{A3}$)--($\mathbf{A4}$), H\"older inequality and Proposition \ref{thmforcontinuity} (i),
  \begin{eqnarray}\label{I1st2}
  	\mathbb{E}|I^{\ep,1}_{ij}(s,t)|^p&=&\mathbb{E}\Big|\sum_k\int_s^t\xei_{s,u}\mi_{jk}(\xe_u)dW^k_u\Big|^p\nonumber\\&\lesssim&\sum_k\mathbb{E}\Big(\int_s^t|\xei_{s,u}\mi_{jk}(\xe_u)|^2du\Big)^{p/2}\nonumber\\&\lesssim&\mathbb{E}\Big(\int_s^t|\xei_{s,u}|^2du\Big)^{p/2}\nonumber\\&\lesssim&(t-s)^{\frac{p}{2}-1}\int_s^t\mathbb{E}|\xei_{s,u}|^pdu\nonumber\\&\lesssim&(t-s)^{\frac{p}{2}-1}\int_s^t (u-s)^{p/2}du\lesssim(t-s)^p,
  \end{eqnarray}
  and
  \begin{eqnarray}\label{I2st2}
  	\mathbb{E}|I^{\ep,2}_{ij}(s,t)|^p&=&\mathbb{E}\Big|\sum_k\int_s^t\xei_{s,u}\mi_{jk}(\xe_u)F_k(\xe_u)du\Big|^p\nonumber\\&\lesssim&\mathbb{E}\Big(\int_s^t|\xei_{s,u}|du\Big)^p\nonumber\\&\leq&(t-s)^{p-1}\int_s^t\mathbb{E}|\xei_{s,u}|^pdu\nonumber\\&\lesssim&(t-s)^{p-1}\int_s^t (u-s)^{p/2}du\nonumber\\&\lesssim&(t-s)^{3p/2}\lesssim(t-s)^p.
  \end{eqnarray}
  By integration-by-part formula and (\ref{SKxy}),
  \begin{eqnarray}\label{preest-I3st2}
  	&&I^{\ep,3}_{ij}(s,t)\nonumber\\&=&-\ep\sum_k\int_s^t\xei_{s,u}\mi_{jk}(\xe_u)d\yek_u\nonumber\\&=&-\ep\sum_k\Big\{\xei_{s,t}\mi_{jk}(\xe_t)\yek_t-\int_s^t\yek_u d[\xei_{s,u}\mi_{jk}(\xe_u)]\Big\}\nonumber\\&=&-\ep\sum_k\xei_{s,t}\mi_{jk}(\xe_t)\yek_t+ \ep\sum_k\int_s^t\yek_u\mi_{jk}(\xe_u)d\xei_u+\ep\sum_{k}\int_s^t\yek_u\xei_{s,u}d[\mi_{jk}(\xe_u)]\nonumber\\&=&-\ep\sum_k\xei_{s,t}\mi_{jk}(\xe_t)\yek_t+\ep \sum_k\int_s^t\yek_u\mi_{jk}(\xe_u)d\xei_u+\ep\sum_{k,l}\int_s^t\yek_u\xei_{s,u}\partial{x_l}\mi_{jk}(\xe_u)d\xel_u\nonumber\\&=&-\ep\sum_k\xei_{s,t}\mi_{jk}(\xe_t)\yek_t+ \sum_k\int_s^t\yek_u\mi_{jk}(\xe_u)\yei_udu+\sum_{k,l}\int_s^t\yek_u\yel_u\xei_{s,u}\partial{x_l}\mi_{jk}(\xe_u)du\nonumber\\&=:&J^{\ep,1}_{ij}(s,t)+J^{\ep,2}_{ij}(s,t)+J^{\ep,3}_{ij}(s,t).
  \end{eqnarray}
  By H\"older inequality, ($\mathbf{A3}$) and Proposition \ref{thmformomentestimate},
  \begin{eqnarray}\label{J1st2}
  	\mathbb{E}|J^{\ep,1}_{ij}(s,t)|^p&\lesssim&\ep^p\mathbb{E}(|\xei_{s,t}|^p|\yek_t|^p)\nonumber\\&\lesssim&\ep^p\sqrt{\mathbb{E}|\xei_{s,t}|^{2p}}\sqrt{\mathbb{E}|\yek_t|^{2p}}\nonumber\\&\lesssim&\ep^p(t-s)^{p/2}\leq(t-s)^{p},
  \end{eqnarray}
  where we have used the dichotomy assumption in the last inequality.
  
  By ($\mathbf{A3}$), H\"older inequality and Proposition \ref{thmformomentestimate},
  \begin{eqnarray}\label{J2st2}
  	\mathbb{E}|J^{\ep,2}_{ij}(s,t)|^p&\lesssim&\mathbb{E}\Big(\int_s^t|\yek_u\yei_u|du\Big)^p\nonumber\\&\leq&(t-s)^{p-1}\int_s^t\mathbb{E}|\yek_u\yei_u|^pdu\lesssim(t-s)^p.
  \end{eqnarray}
  By ($\mathbf{A3}$) and Proposition \ref{thmformomentestimate},
  \begin{eqnarray}\label{J3st2}
  	\mathbb{E}|J^{\ep,3}_{ij}(s,t)|^p&\lesssim&\mathbb{E}\Big(\int_s^t|\yek_u\yel_u\xei_{s,u}|du\Big)^p\nonumber\\&\leq&(t-s)^{p-1}\int_s^t\mathbb{E}|\yek_u\yel_u\xei_{s,u}|^pdu\nonumber\\&\leq&(t-s)^{p-1}\int_s^t\sqrt{\mathbb{E}|\yek_u\yel_u|^{2p}}\sqrt{\mathbb{E}|\xei_{s,u}|^{2p}}du\nonumber\\&\lesssim&(t-s)^{p-1}\int_s^t(u-s)^{p/2}du\nonumber\\&\lesssim&(t-s)^{3p/2}\lesssim(t-s)^{p}.
  \end{eqnarray}
  Combining (\ref{preest-I3st2})--(\ref{J3st2}),
  \begin{equation}\label{I3st2}
  	\mathbb{E}|I^{\ep,3}_{ij}(s,t)|^p\lesssim|t-s|^p.
  \end{equation}
  Combining (\ref{preest-increment2})--(\ref{I2st2}) and (\ref{I3st2}),
  \begin{equation}
  	\mathbb{E}|\mathbb{X}^\ep_{s,t}|^p\lesssim|t-s|^p.
  \end{equation} 
  
	\end{proof}
	An immediate consequence of Proposition \ref{thmforcontinuity} and Kolmogorov criterion for rough paths \cite[Theorem 3.1]{FH20} is the following corollary.
	\begin{coro}\label{corononame}
		Under assumptions ($\mathbf{A1}$)--($\mathbf{A4}$), for all $1\leq p<\infty$ and $\frac{1}{3}<\al<\frac{1}{2},$
		$$\supe\mathbb{E}||\xe||_{\al}^p<\infty,\quad \supe\mathbb{E}||\mathbb{X}^\ep||_{2\al}^p<\infty.$$
	\end{coro}
	\begin{remark}
	 (i) From Corollary \ref{corononame} and Kolmogorov tightness criterion for rough paths\cite[Theorem 3.10]{FH20}, we see that $(\mathbf{X}^\ep)_{0<\ep\leq 1}$ is tight in $\mathscr{C}^{0,\al}([0,T];\mathbb{R}^d).$
	 
	 (ii) Since for each $f\in C^\al([0,T];\mathbb{R}^d),$ $||f||_\infty\leq ||f||_\al T^\al+|f(0)|$, we automatically have
	 \begin{equation}\label{supnormforxe}
	 	\supe\mathbb{E}\supt|\xe_t|^p<\infty,\quad p\geq 1.
	 \end{equation}
	\end{remark}
	
\begin{remark}
	One might expect to show Proposition \ref{thmforcontinuity} (i) by integrating factor method in classical ODE theory, that is, write
	$$\sxm_t=\frac{1}{\ep^2}\intot\intos \zr_s dW_rds+\frac{1}{\ep^2}\intot\intos \zr_s F(\sxm_r) drds,$$
	where  $\zr_s=\phie_s(\phie_r)^{-1},\enspace  0\leq r\leq s\leq T$, and $\phi^\ep$ is the solution of the random ODE $$\dot{\phi}^\ep_t=-\frac{1}{\ep^2}M(\sxm_t)\phie_t,\enspace \phie_0=\mathrm{I}_d;$$
	 see \cite{SWW24} for detailed computation. However, this seems incorrect, since the integrand $r\mapsto\zr_s$ is not $\mathcal{F}_r$-adapted, let alone progressively measurable, so the integrand with respect to Brownian motion is not well-defined. This is also pointed out by Blassel~\cite{arxiv2602}.
	
\end{remark}
\section{Averaging Principle}\label{section-aver}
In this section, we explore the averaging principle for system (\ref{SKxy}). To be precise, we establish the following proposition.
\begin{prop}\label{propforaveraging}
	Let f be the form of either
	\begin{equation}\label{formoff}
		f(x,y)=x_iy_ky_lg(x),\quad x=(x_1,...,x_d)\in\mathbb{R}^d,\enspace y=(y_1,...,y_d)\in\mathbb{R}^d,
	\end{equation}
	or
	\begin{equation}\label{formoff2}
		f(x,y)=y_ky_lg(x),\quad x=(x_1,...,x_d)\in\mathbb{R}^d,\enspace y=(y_1,...,y_d)\in\mathbb{R}^d,
	\end{equation}
	where $g\in C^2(\mathbb{R}^d; \mathbb{R})$ is Lipschitz. Under assumptions ($\mathbf{A1}$)--($\mathbf{A4}$), for each $1\leq p<\infty,$
	$$\lim\limits_{\ep\to0}\mathbb{E}\Big|\intot f(\xe_s,\ye_s)-\bar{f}(\xe_s)ds\Big|^p=0,$$
	where $\bar{f}(x):=\int f(x,y)\nu^x(dy)$ and $\nu^x$ is the unique invariant measure of the following system:
	\begin{equation}\label{frozenprocess}
		dY^x_t=-M(x)Y^x_tdt+dW_t.
	\end{equation}
\end{prop}
\begin{remark}
	One sees that the non-linearity $F$ vanishes in (\ref{frozenprocess}). Intuitively speaking, the reason is that, the fast system in (\ref{SKxy}) with frozen slow variable $$dY^{\ep,x}_t=-\frac{1}{\ep^2}M(x)Y^{\ep,x}_tdt+\frac{1}{\ep}F(x)dt+\frac{1}{\ep}dW_t,$$
	becomes
	\begin{equation}\label{tildeyex}
		d\tilde{Y}^{\ep,x}_t=-M(x)\tilde{Y}^{\ep,x}_tdt+\ep F(x)dt+d\tilde{W}_t
	\end{equation}
	after change-of-variable $\tilde{Y}^{\ep,x}_t:=Y^{\ep,x}_{\ep^2 t},$ where $\tilde{W}_t:=\frac{1}{\ep}W_{\ep^2 t}\overset{d}{=}W_t,$ and the term $\ep F(x)$ in (\ref{tildeyex}) disappears as $\ep$ goes to $0$.
\end{remark}
\begin{remark}\label{remarkforbarf}
	It is known that \cite[Proposition 3.5]{PAV14}, under assumption ($\mathbf{A1}$), equation (\ref{frozenprocess}) admits a unique invariant measure which is Gaussian: $\nu^x=\mathcal{N}(0,J(x))$, where $J$ is the solution of (\ref{lyapunov}). Consequently,
	\begin{equation}\label{explicitforbarf}
		\bar{f}(x)=x_ig(x)J_{kl}(x)
	\end{equation}
	if $f$ is of the form (\ref{formoff}) and
	\begin{equation}
		\bar{f}(x)=g(x)J_{kl}(x)
	\end{equation}
	if $f$ is of the form (\ref{formoff2}).
\end{remark}
\begin{remark}\label{remarkforlyapunov}
	It is known that \cite[Page 179]{BELLMAN}, under assumption ($\mathbf{A1}$), the Lyapunov equation
	$$M(x)A(x)+A(x)M^\top(x)=B(x)$$ admits a unique solution
	$$A(x)=-\int_0^\infty e^{-M^\top(x)t}B(x)e^{-M(x)t}dt.$$
	This in particular implies that
	\begin{eqnarray*}
		|A(x)|\leq\int_0^\infty |e^{-M^\top(x)t}|\cdot|B(x)|\cdot|e^{-M(x)t}|dt\leq\int_0^\infty e^{-2\lambda t}|B(x)|dt=\frac{1}{2\lambda}|B(x)|.
	\end{eqnarray*} 
	Therefore,  $A$ has polynomial growth whenever  $B$ does.
\end{remark}

\noindent\emph{Proof of Proposition \ref{propforaveraging}:}

 We only prove this proposition for $f$ of the form (\ref{formoff}), since the other case is similar and easier. For each $x\in\mathbb{R}^d$, denote the generator of SDE (\ref{frozenprocess}) by $\mathcal{L}^x,$ so that
\begin{equation}\label{generator}
	\mathcal{L}^x\phi(x,y)=(D_y\phi(x,y),-M(x)y)+\frac{1}{2}\mathrm{Tr}(D_{yy}\phi(x,y)),\quad\phi\in C^{2}(\mathbb{R}^d\times\mathbb{R}^d;\mathbb{R}).
\end{equation} Consider Poisson equation
\begin{equation}\label{poisson}
	-\mathcal{L}^x\phi(x,y)=f(x,y)-\bar{f}(x).
\end{equation}
From (\ref{explicitforbarf}) it is straightforward to check that
\begin{equation}\label{solutionofpoisson}
	\phi(x,y)=x_i g(x)(y^\top A(x)y-\mathrm{Tr}[A(x)J(x)]),
\end{equation}
where $J$ is the solution of (\ref{lyapunov}) and $A$ is the solution of the following Lyapunov equation
\begin{equation}\label{lyapunov2}
	M^\top(x)A(x)+A(x)M(x)=\frac{1}{2}(e_ke_l^\top+e_le_k^\top).
\end{equation}
 Here, $(e_n)_{1\leq n\leq d}$ is the orthonormal basis of $\mathbb{R}^d.$  As a consequence of ($\mathbf{A1}$)--($\mathbf{A3}$), $\phi\in C^{2}(\mathbb{R}^d\times\mathbb{R}^d;\mathbb{R}).$
By It\^o formula and (\ref{generator}),
\begin{eqnarray}
	&&d\phi(\xe_t,\ye_t)\nonumber\\&=&\Big(D_x\phi(\xe_t,\ye_t),\frac{1}{\ep}\ye_t\Big)dt+\Big(D_y\phi(\xe_t,\ye_t),-\frac{1}{\ep^2}M(\xe_t)\ye_t+\frac{1}{\ep}F(\xe_t)\Big)dt\nonumber\\&&+\frac{1}{2}\mathrm{Tr}\Big[\frac{1}{\ep^2}D^2_{yy}\phi(\xe_t,\ye_t)\Big]dt+\Big(\frac{1}{\ep}D_y\phi(\xe_t,\ye_t),dW_t\Big)\nonumber\\&=&\frac{1}{\ep}(D_x\phi(\xe_t,\ye_t),\ye_t)dt+\frac{1}{\ep}(D_y\phi(\xe_t,\ye_t),F(\xe_t))dt\\\nonumber&&+\frac{1}{\ep}(D_y\phi(\xe_t,\ye_t),dW_t)+\frac{1}{\ep^2}\mathcal{L}^{\xe_t}\phi(\xe_t,\ye_t)dt.
\end{eqnarray}
Integrating on $[0,t]$, rearranging and using (\ref{poisson}),
\begin{eqnarray}
	&&\intot f(\xe_s,\ye_s)-\bar{f}(\xe_s) ds\nonumber\\&=&-\intot\mathcal{L}^{\xe_s}\phi(\xe_s,\ye_s)ds\nonumber\\&=&-\ep^2[\phi(\xe_t,\ye_t)-\phi(0,0)]+\ep\intot(D_x\phi(\xe_s,\ye_s),\ye_s)ds\nonumber\\&&+\ep\intot(D_y\phi(\xe_s,\ye_s),F(\xe_s))ds+\ep\intot(D_y\phi(\xe_s,\ye_s),dW_s).
\end{eqnarray}
With the aid of Proposition \ref{thmformomentestimate}, the proof is finished by checking that $\phi$, $\phi_x$ and $\phi_y$ are of at most polynomial growth, that is, there exist constants $C>0$ and $q\geq 1$, such that for all $x,y\in\mathbb{R}^d$,
$$\max\{|\phi(x,y)|,|\phi_x(x,y)|,|\phi_y(x,y)|\}\leq C(1+|x|^q+|y|^q).$$
Differentiating on both sides of (\ref{solutionofpoisson}), 
\begin{equation}\label{phiy}
	\phi_y(x,y)=2x_i g(x)A(x)y,
\end{equation}
and
\begin{eqnarray}\label{phix}
	&&\phi_{x_j}(x,y)\nonumber\\&=&[\delta_{ij}g(x)+x_i\partial_{x_j}g(x)]\times[y^\top A(x)y-\mathrm{Tr}(A(x)J(x))]\nonumber\\&&+x_ig(x)\Big\{y^\top[\partial_{x_j}A(x)]y-\mathrm{Tr}[(\partial_{x_j}A(x))J(x)+A(x)\partial_{x_j}J(x)]\Big\}.
\end{eqnarray}
Differentiating on both sides of (\ref{lyapunov}) and (\ref{lyapunov2}), one finds that $\partial_{x_j}J(x)$ and $\partial_{x_j}A(x)$ satisfy Lyapunov equations
\begin{equation}
	M(x)\partial_{x_j}J(x)+\partial_{x_j}J(x)M^\top(x)=-([\partial_{x_j} M(x)]J(x)+J(x)[\partial_{x_j}M^\top(x)]).
\end{equation}
and
\begin{equation}
	M^\top(x)\partial_{x_j}A(x)+\partial_{x_j}A(x)M(x)=-([\partial_{x_j} M^\top(x)]A(x)+A(x)[\partial_{x_j}M(x)]),
\end{equation}
so the polynomial growth property is obtained by combining (\ref{solutionofpoisson}), (\ref{phiy}), (\ref{phix}) and Remark \ref{remarkforlyapunov}.

\qed

\section{Proof of Main Theorems}\label{section-proof}
With preparation above, we prove the main theorems now.

\noindent\emph{Proof of Theorem \ref{maintheorem}:}

By Corollary \ref{corononame} and interpolation theorem \cite[Theorem A.15]{FV10}, it suffices to prove the pointwise convergence, that is,
$$\xe_t\to X_t,\quad \text{in}\enspace L^p(\Omega,\mathcal{F},P)$$ and 
$$\mathbb{X}^\ep_{0,t}\to \mathbb{X}_{0,t},\quad \text{in}\enspace L^p(\Omega,\mathcal{F},P)$$
for each $0\leq t\leq T.$

From (\ref{expressionfordxe})
\begin{eqnarray}\label{expressxe}
	\xe_t&=&\Big[\intot\mi(\xe_s)dW_s+\intot\mi(\xe_s)F(\xe_s)ds\Big]-\ep\intot\mi(\xe_s)d\ye_s\nonumber\\&=:&\mathbf{I}(\ep,t)-\mathbf{J}(\ep,t).
\end{eqnarray}
Plainly,
\begin{eqnarray}\label{expressii}
	\mathbf{I}^i(\ep,t)=\sum_k\intot\mi_{ik}(\xe_s)dW^k_s+\sum_k\intot\mi_{ik}(\xe_s)F_k(\xe_s)ds.
\end{eqnarray}
From (\ref{preestforI3st}) we see that
\begin{eqnarray}\label{expressji}
	\mathbf{J}^i(\ep,t)=\ep\sum_k\mi_{ik}(\xe_t)\yek_t-\sum_{k,l}\intot\yek_s\yel_s\partial_{x_l}\mi_{ik}(\xe_s)ds.
\end{eqnarray}
Combining (\ref{expressxe}), (\ref{expressii}) and (\ref{expressji}),
\begin{eqnarray*}
	&&\xei_t\nonumber\\&=&\sum_k\intot\mi_{ik}(\xe_s)dW^k_s+\sum_k\intot\mi_{ik}(\xe_s)F_k(\xe_s)ds\nonumber\\&&+\sum_{k,l}\intot\yek_s\yel_s\partial_{x_l}\mi_{ik}(\xe_s)ds-\ep\sum_k\mi_{ik}(\xe_t)\yek_t.
\end{eqnarray*}
From (\ref{level1limit}),
$$X^i_t=\sum_k\intot\mi_{ik}(X_s)dW^k_s+\sum_k\intot\mi_{ik}(X_s)F_k(X_s)ds+\sum_{k,l}\intot J_{kl}(X_s)\partial_{x_l}\mi_{ik}(X_s)ds,$$
so
\begin{eqnarray}\label{esxexi}
	&&\xei_t-X^i_t\nonumber\\&=&\sum_k\intot[\mi_{ik}(\xe_s)-\mi_{ik}(X_s)]dW^k_s+\sum_k\intot\mi_{ik}(\xe_s)F_k(\xe_s)-\mi_{ik}(X_s)F_k(X_s)ds\nonumber\\&&-\ep\sum_k\mi_{ik}(\xe_t)\yek_t+\sum_{k,l}\intot[\yek_s\yel_s\partial_{x_l}\mi_{ik}(\xe_s)-J_{kl}(\xe_s)\partial_{x_l}\mi_{ik}(\xe_s)]ds\nonumber\\&&+\sum_{k,l}\intot [J_{kl}(\xe_s)\partial_{x_l}\mi_{ik}(\xe_s)-J_{kl}(X_s)\partial_{x_l}\mi_{ik}(X_s)]ds\nonumber\\&=:&\mathbf{K}_1(\ep,t)+\mathbf{K}_2(\ep,t)+\mathbf{K}_3(\ep,t)+\mathbf{K}_4(\ep,t)+\mathbf{K}_5(\ep,t).
\end{eqnarray}
By Burkholder--Davis--Gundy inequality, ($\mathbf{A3}$) and H\"older inequality
\begin{eqnarray}\label{esk1}
	&&\mathbb{E}|\mathbf{K}_1(\ep,t)|^p\nonumber\\&\lesssim&\sum_k\mathbb{E}\Big|\intot[\mi_{ik}(\xe_s)-\mi_{ik}(X_s)]dW^k_s\Big|^p\nonumber\\&\lesssim&\sum_k\mathbb{E}\Big(\intot |\mi_{ik}(\xe_s)-\mi_{ik}(X_s)|^2ds\Big)^{p/2}\nonumber\\&\lesssim&\intot\mathbb{E}|\xe_s-X_s|^pds.
\end{eqnarray}
Assumptions ($\mathbf{A3}$) and ($\mathbf{A4}$) imply that
\begin{eqnarray*}
	&&|\mathbf{K}_2(\ep,t)|\nonumber\\&\leq&\sum_k\intot|\mi_{ik}(\xe_s)|\cdot|F_k(\xe_s)-F_k(X_s)|+|F_k(X_s)|\cdot|\mi_{ik}(\xe_s)-\mi_{ik}(X_s)|ds\nonumber\\&\lesssim&\intot|\xe_s-X_s|ds.
\end{eqnarray*}
By H\"older inequality,
\begin{eqnarray}\label{esk2}
	\mathbb{E}|\mathbf{K}_2(\ep,t)|^p\lesssim\mathbb{E}\intot |\xe_s-X_s|^pds=\intot\mathbb{E}|\xe_s-X_s|^pds.
\end{eqnarray}
By the same argument,
\begin{eqnarray}\label{esk5}
	\mathbb{E}|\mathbf{K}_5(\ep,t)|^p\lesssim\intot\mathbb{E}|\xe_s-X_s|^pds.
\end{eqnarray}
Utilizing Proposition \ref{thmformomentestimate} and ($\mathbf{A3}$),
\begin{eqnarray}\label{esk3}
	\mathbb{E}|\mathbf{K}_3(\ep,t)|^p\lesssim\ep^p\sum_k\supt\mathbb{E}|\yek_t|^p\lesssim\ep^p.
\end{eqnarray}
Lastly, from Proposition \ref{propforaveraging} and Remark \ref{remarkforbarf}, there exists a function $O(\ep,t)$ such that
\begin{eqnarray}\label{esk4}
	\mathbb{E}|\mathbf{K}_4(\ep,t)|^p\lesssim O(\ep,t),
\end{eqnarray}
and that for each $t\in [0,T],$
$$\lim\limits_{\ep\to0} O(\ep,t)=0.$$
Combining (\ref{esxexi})--(\ref{esk4}), we see that
$$\mathbb{E}|\xei_t-X^i_t|^p\lesssim\intot\mathbb{E}|\xe_s-X_s|^pds+\ep^p+O(\ep,t).$$
Summing up with respect to $i$,
$$\mathbb{E}|\xe_t-X_t|^p\lesssim\intot\mathbb{E}|\xe_s-X_s|^pds+\ep^p+O(\ep,t).$$
By Gronwall inequality,
\begin{equation}\label{resultforlevel1convergence}
	\lim\limits_{\ep\to 0}\mathbb{E}|\xe_t-X_t|^p=0.
\end{equation}
This shows the convergence of $\xe$. By Proposition \ref{thmformomentestimate} and dominated convergence theorem,
\begin{equation}
	\lim\limits_{\ep\to 0}\mathbb{E}\Big|\intot\bar{f}(\xe_s)-\bar{f}(X_s)ds\Big|^p=0.
\end{equation}
In conjunction with Proposition \ref{propforaveraging},
\begin{equation}\label{averagingresult}
	\lim\limits_{\ep\to 0}\mathbb{E}\Big|\intot f(\xe_s,\ye_s)-\bar{f}(X_s)ds\Big|^p=0.
\end{equation}

Next, we turn to show the convergence of $\mathbb{X}^\ep.$ By (\ref{expressionfordxe}),
\begin{eqnarray}\label{level2path}
	&&\intot\xe_s\otimes d\xe_s\nonumber\\&=&\Big(\intot\xe_s\otimes[\mi(\xe_s)dW_s]+\intot\xe_s\otimes[\mi(\xe_s)F(\xe_s)]ds\Big)-\ep\intot\xe_s\otimes[\mi(\xe_s)d\ye_s]\nonumber\\&=:&\mathbf{A}(\ep,t)-\mathbf{B}(\ep,t).
\end{eqnarray}
We explore the limit of (\ref{level2path}) as $\ep\to 0$ componentwise. For $\mathbf{A}(\ep,t),$ by (\ref{resultforlevel1convergence}), Burkholder--Davis--Gundy inequality and dominated convergence theorem with the help of (\ref{supnormforxe}),
\begin{eqnarray}\label{aconvergence}
	&&\mathbf{A}^{ij}(\ep,t)\nonumber\\&=&\Big(\intot\xe_s\otimes[\mi(\xe_s)dW_s]\Big)_{ij}+\Big(\intot\xe_s\otimes[\mi(\xe_s)F(\xe_s)]ds\Big)_{ij}\nonumber\\&=&\intot\xei_s\times[\mi(\xe_s)dW_s]^j+\intot\xei_s\times[\mi(\xe_s)F(\xe_s)ds]^j\nonumber\\&=&\sum_k\intot\xei_s\mi_{jk}(\xe_s)dW^k_s+\sum_k\intot\xei_s\mi_{jk}(\xe_s)F_k(\xe_s)ds\nonumber\\&\to&\sum_k\intot X^i_s\mi_{jk}(X_s)dW^k_s+\sum_k\intot X^i_s\mi_{jk}(X_s)F_k(X_s)ds,
\end{eqnarray}
where the limit is in $L^p(\Omega,\mathcal{F},P).$ For $\mathbf{B}(\ep,t),$ we see from (\ref{preest-I3st2}) that
\begin{eqnarray}
	&&\mathbf{B}^{ij}(\ep,t)\nonumber\\&=&\ep\sum_k\xei_t\mi_{jk}(\xe_t)\yek_t-\sum_k\intot\yek_s\yei_s\mi_{jk}(\xe_s)ds-\sum_{k,l}\intot\yek_s\yel_s\xei_{s}\partial_{x_l}\mi_{jk}(\xe_s)ds\nonumber\\&=:&\mathbf{B}^{ij}_1(\ep,t)-\mathbf{B}^{ij}_2(\ep,t)-\mathbf{B}^{ij}_3(\ep,t).
\end{eqnarray}
By Proposition \ref{thmformomentestimate},  (\ref{supnormforxe}) and ($\mathbf{A3}$), $(\xei_t\mi_{jk}(\xe_t)\yek_t)_{0<\ep\leq 1}$ is bounded in $L^p(\Omega,\mathcal{F},P)$ for each $t\in [0,T],$ so that
\begin{equation}\label{b1convergence}
	\mathbf{B}_1^{ij}(\ep,t)\to 0 \quad\text{in}\quad L^p(\Omega,\mathcal{F},P).
\end{equation}
By (\ref{averagingresult}),
\begin{equation}\label{b21convergence}
	\mathbf{B}_{2}^{ij}(\ep,t)\to \sum_k\intot\mi_{jk}(X_s)J_{ik}(X_s)ds,
\end{equation}
and
\begin{equation}\label{b22convergence}
	\mathbf{B}_{3}^{ij}(\ep,t)\to \sum_{k,l}\intot X^i_s(\partial_{x_l}\mi_{jk})(X_s)J_{kl}(X_s)ds,
\end{equation}
where the two limits above are both in $L^p(\Omega,\mathcal{F},P).$
Combining (\ref{aconvergence}), (\ref{b1convergence}), (\ref{b21convergence}) and (\ref{b22convergence}), we conclude that
\begin{eqnarray}\label{componentconvergence}
	&&\Big(\intot\xe_s\otimes d\xe_s\Big)_{ij}\nonumber\\&\to&\sum_k\intot X^i_s\mi_{jk}(X_s)dW^k_s+\sum_k\intot X^i_s\mi_{jk}(X_s)F_k(X_s)ds\nonumber\\&&+\sum_k\intot\mi_{jk}(X_s)J_{ik}(X_s)ds+\sum_{k,l}\intot X^i_s(\partial_{x_l}\mi_{jk})(X_s)J_{kl}(X_s)ds,
\end{eqnarray}
in $L^p(\Omega,\mathcal{F},P).$ 

The proof is finished after verifying that the right hand side of (\ref{componentconvergence}) coincides with $\mathbb{X}^{ij}_{0,t},$ where $\mathbb{X}$ is defined via (\ref{defofbbx}). Let us calculate $\mathbb{X}^{ij}_{0,t}.$ Writing (\ref{level1limit}) componentwise, we see that
\begin{eqnarray*}
	&&\intot X^i_s dX^j_s\\&=&\sum_{k,l}\intot X^i_s (\partial_{x_l}\mi_{jk})(X_s)J_{kl}(X_s)ds+\sum_k\intot X^i_s\mi_{jk}(X_s)F_k(X_s)ds+\sum_k X^i_s\mi_{jk}(X_s)dW^k_s,
\end{eqnarray*}
and
\begin{eqnarray*}
	[X^i,X^j]_t=\sum_k\intot\mi_{ik}(X_s)\mi_{jk}(X_s)ds.
\end{eqnarray*}
However, multipling both sides of (\ref{lyapunov}) by $\mi$ on the left and $(\mi)^\top$ on the right handside and then writing it componentwise, we see that
$$\sum_k\intot\mi_{ik}(X_s)\mi_{jk}(X_s)ds=\sum_k\intot\mi_{ik}(X_s)J_{kj}(X_s)+J_{ik}(X_s)\mi_{jk}(X_s)ds.$$
Combining three equalities above,
\begin{eqnarray}
	&&\intot X^i_s\circ dX^j_s\nonumber\\&=&\sum_{k,l}\intot X^i_s (\partial_{x_l}\mi_{jk})(X_s)J_{kl}(X_s)ds+\sum_k\intot X^i_s\mi_{jk}(X_s)F_k(X_s)ds\nonumber\\&&+\sum_k\intot X^i_s\mi_{jk}(X_s)dW^k_s+\frac{1}{2}\sum_k\intot\mi_{ik}(X_s)J_{kj}(X_s)+J_{ik}(X_s)\mi_{jk}(X_s)ds\nonumber.
\end{eqnarray}
Plainly,
\begin{eqnarray*}
	&&\frac{1}{2}\Big(\intot J(X_s)(\mi)^\top(X_s)-\mi(X_s)J(X_s)ds\Big)_{ij}\\&=&\frac{1}{2}\sum_k\intot J_{ik}(X_s)\mi_{jk}(X_s)-\mi_{ik}(X_s)J_{kj}(X_s)ds.
\end{eqnarray*}
Therefore,
\begin{eqnarray}\label{finalresult}
	&&\Big(\intot X_s\otimes\circ dX_s\Big)_{ij}+\frac{1}{2}\Big(\intot J(X_s)(\mi)^\top(X_s)-\mi(X_s)J(X_s)ds\Big)_{ij}\nonumber\\&=&\sum_{k,l}\intot X^i_s (\partial_{x_l}\mi_{jk})(X_s)J_{kl}(X_s)ds+\sum_k\intot X^i_s\mi_{jk}(X_s)F_k(X_s)ds\nonumber\\&&+\sum_k\intot X^i_s\mi_{jk}(X_s)dW^k_s+\sum_k\intot J_{ik}(X_s)\mi_{jk}(X_s)ds,
\end{eqnarray}
which coincides with (\ref{componentconvergence}).\qed

Before demonstrating Theorem \ref{maintheorem2}, we discuss a slightly more general result about transferring RDE to SDE. Consider the following RDE $$dZ_t=b(Z_t)dt+\sigma(Z_t)d\mathbf{Y}_t,$$ where $\mathbf{Y}=(Y,\mathbb{Y})$ is a $\mathscr{C}^\al([0,T];\mathbb{R}^d)$-valued random element, $Y$ is a continuous semimartingale,  $$\mathbb{Y}^{ij}_{s,t}=\mathbb{Y}^{\text{It\^o},ij}_{s,t}+\int_s^t a^{ij}_udu,\quad\text{and}\quad \mathbb{Y}^{\text{It\^o},ij}_{s,t}:=\int_s^t Y^i_{s,u}dY^j_u,$$ the later integral being It\^o integral, and $a$ is a $C([0,T];\mathbb{R}^{d\times d})$-valued random element. Then for almost every sample path, we have
\begin{eqnarray}\label{generalrecover1}
	&&Z_t-Z_s\nonumber\\&=&b(Z_s)(t-s)+\sigma(Z_s)Y_{s,t}+D\sigma(Z_s)\sigma(Z_s)\mathbb{Y}_{s,t}+R_{s,t},
\end{eqnarray}
with $\lim\limits_{|\mathcal{P}|\to 0}\sup\limits_{\mathcal{P}\subset [0,T]}\sum\limits_{[s,t]\in\mathcal{P}}|R_{s,t}|=0,$ $P$-$\text{a.s.}$\,.
But
\begin{eqnarray}\label{generalrecover2}
	&&(D\sigma(Z_s)\sigma(Z_s)\mathbb{Y}_{s,t})_i\nonumber\\&=&\sum_{j,k,l}\partial_{z_k}\sigma_{ij}(Z_s)\sigma_{kl}(Z_s)\mathbb{Y}^{lj}_{s,t}\nonumber\\&=&\sum_{j,k,l}\partial_{z_k}\sigma_{ij}(Z_s)\sigma_{kl}(Z_s)\mathbb{Y}^{\text{It\^o},lj}_{s,t}+\sum_{j,k,l}\partial_{z_k}\sigma_{ij}(Z_s)\sigma_{kl}(Z_s)\int_s^t a^{lj}_udu\nonumber\\&=&\sum_{j,k,l}\partial_{z_k}\sigma_{ij}(Z_s)\sigma_{kl}(Z_s)\mathbb{Y}^{\text{It\^o},lj}_{s,t}+\sum_{j,k,l}\partial_{z_k}\sigma_{ij}(Z_s)\sigma_{kl}(Z_s)[a^{lj}_s(t-s)+\hat{R}_{s,t}],
\end{eqnarray}
where $\lim\limits_{|\mathcal{P}|\to 0}\sup\limits_{\mathcal{P}\subset [0,T]}\sum\limits_{[s,t]\in\mathcal{P}}|\hat{R}_{s,t}|=0,$ $P$-$\text{a.s.}$\,. Combining (\ref{generalrecover1}) and (\ref{generalrecover2}),
\begin{eqnarray}
	&&(Z_t-Z_s)_i\nonumber\\&=&[b_i(Z_s)+\sum_{j,k,l}\partial_{z_k}\sigma_{ij}(Z_s)\sigma_{kl}(Z_s)a^{lj}_s](t-s)\nonumber\\&&+\sum_{j}\sigma_{ij}(Z_s)Y^j_{s,t}+\sum_{j,k,l}\partial_{z_k}\sigma_{ij}(Z_s)\sigma_{kl}(Z_s)\mathbb{Y}^{\text{It\^o},lj}_{s,t}+\tilde{R}^i_{s,t},\nonumber
\end{eqnarray}
where $\lim\limits_{|\mathcal{P}|\to 0}\sup\limits_{\mathcal{P}\subset [0,T]}\sum\limits_{[s,t]\in\mathcal{P}}|\tilde{R}^i_{s,t}|=0,$ $P$-$\text{a.s.}$\,.
By equivalence theorem of RDE and SDE\cite[Proposition 6.9]{FZ18}, for each $1\leq i\leq d,$
\begin{equation}\label{generalresult}
	dZ^i_t=[b_i(Z_t)+\sum\limits_{j,k,l}\partial_{z_k}\sigma_{ij}(Z_t)\sigma_{kl}(Z_t)a^{lj}_t]dt+\sum_{j}\sigma_{ij}(Z_t)dY^j_t.
\end{equation}

\noindent\emph{Proof of Theorem \ref{maintheorem2}:}

By Theorem \ref{maintheorem} and a probabilistic version of continuity of It\^o--Lyons map\cite[Theorem 6.1]{FZ18}, $Z^\ep\to Z$ in $C^\alpha([0,T];\mathbb{R}^d)$ in probability, with 
$$dZ_t=b(Z_t)dt+\sigma(Z_t)d\mathbf{X}_t\,.$$
 Now we transfer this RDE to SDE. By the same argument leading to (\ref{finalresult}), we have
\begin{eqnarray}\label{finalresult2}
	&&\mathbb{X}^{ij}_{s,t}\nonumber\\&=&\Big(\int_s^t X_u\otimes\circ dX_u\Big)_{ij}+\frac{1}{2}\Big(\int_s^t J(X_u)(\mi)^\top(X_u)-\mi(X_u)J(X_u)du\Big)_{ij}\nonumber\\&=&\sum_{k,l}\int_s^t X^i_u S_j(X_u)du+\sum_k\int_s^t X^i_u\mi_{jk}(X_u)F_k(X_u)du\nonumber\\&&+\sum_k\int_s^t X^i_u\mi_{jk}(X_u)dW^k_u+\sum_k\int_s^t J_{ik}(X_u)\mi_{jk}(X_u)du\nonumber\\&=&\mathbb{X}^{\text{It\^o},ij}_{s,t}+\sum_k\int_s^t J_{ik}(X_u)\mi_{jk}(X_u)du.
\end{eqnarray}
Applying (\ref{generalresult}),
	\begin{eqnarray}
	&&dZ^i_t\nonumber\\&=&[b_i(Z_t)+\sum\limits_{j,k,l,k^\prime}\partial_{z_k}\sigma_{ij}(Z_t)\sigma_{kl}(Z_t)J_{lk^\prime}(X_t)\mi_{jk^\prime}(X_t)]dt+\sum_{j}\sigma_{ij}(Z_t)dY^j_t\nonumber\\&=&[b_i(Z_t)+\sum\limits_{j,k,l,k^\prime}\partial_{z_k}\sigma_{ij}(Z_t)\sigma_{kl}(Z_t)J_{lk^\prime}(X_t)\mi_{jk^\prime}(X_t)]dt\nonumber\\&&+[\sum_{j}\sigma_{ij}(Z_t)S_j(X_t)+\sum_{j,k}\sigma_{ij}(Z_t)M^{-1}_{jk}(X_t)F_k(X_t)]dt+\sum_{j,k}\sigma_{ij}(Z_t)M^{-1}_{jk}(X_t)dW^k_t.\nonumber
	\end{eqnarray}\qed

	\bibliographystyle{abbrv}
\end{document}